\documentclass[12pt]{article}
\usepackage{a4wide,amsmath,amssymb,latexsym,amsthm,authblk}
\usepackage[para]{footmisc}

\usepackage[capitalise]{cleveref}
\crefname{equation}{}{}

\newtheorem{theorem}{Theorem}
\newtheorem{lemma}{Lemma}
\newtheorem{proposition}{Proposition}

\newtheorem{corollary}{Corollary}

\theoremstyle{remark}
\theoremstyle{remark}\newtheorem{example}{Example}

\usepackage{tikz}
\usepackage[sort,nocompress]{cite} 

\def\RR{\mathbb{R}}

\def\XX{\mathbb{X}}
\def\YY{\mathbb{Y}}

\def\dist{{\rm dist}}
\def\calA{{\cal A}}

\def\calI{{\cal I}}
\def\calJ{{\cal J}}
\def\calK{{\cal K}}
\def\bfa{\mathbf{a}}
\def\bfc{\mathbf{c}}
\def\bfe{\mathbf{e}}

\def\bfu{\mathbf{u}}
\def\bfv{\mathbf{v}}
\def\bfw{\mathbf{w}}
\def\bfx{\mathbf{x}}
\def\bfy{\mathbf{y}}
\def\bfz{\mathbf{z}}
\def\bfs{\mathbf{s}}
\def\bft{\mathbf{t}}

\newcommand{\spn}{{\rm span}}
\newcommand{\extreme}[1]{\hat{#1}}

\begin{document}

\title{Best low-rank approximations and Kolmogorov $n$-widths}
\author[1]{Michael S. Floater\thanks{michaelf@math.uio.no}}
\author[2]{Carla Manni\thanks{manni@mat.uniroma2.it}}
\author[2]{Espen Sande\thanks{sande@mat.uniroma2.it}}
\author[2]{\\ Hendrik Speleers\thanks{speleers@mat.uniroma2.it}}

\affil[1]{\small Department of Mathematics, University of Oslo, Norway}
\affil[2]{\small Department of Mathematics, University of Rome Tor Vergata, Italy}
\maketitle

\begin{abstract}
We relate the problem of best low-rank approximation in the spectral norm for a matrix $A$ to Kolmogorov $n$-widths and corresponding optimal spaces. We characterize all the optimal spaces for the image of the Euclidean unit ball under $A$ and we show that any orthonormal basis in an $n$-dimensional optimal space generates a best rank-$n$ approximation to $A$. We also present a simple and explicit construction to obtain a sequence of optimal $n$-dimensional spaces once an initial optimal space is known. This results in a variety of solutions to the best low-rank approximation problem and provides alternatives to the truncated singular value decomposition. This variety can be exploited to obtain best low-rank approximations with problem-oriented properties.
\end{abstract}


\section{Introduction}\label{sec:introduction}
The problem of approximating a given matrix by another matrix of a lower rank is labeled as the problem of low-rank approximation (of matrices). It aims to obtain a more compact representation
of data with limited loss of information. Low-rank approximation of matrices is ubiquitous in applications: discretization of partial differential equations, principal component analysis, image processing, data mining, and machine learning, to name a few; see, e.g., \cite{Kumar:2017} for a survey. In particular, it plays an important role in matrix completion \cite{Cai:2010}, which finds in the so-called \emph{Netflix problem} one of its most well-known applications \cite{Hallinan:2016}.

In this paper we consider the classical problem of \emph{best} low-rank approximation of matrices measured in the spectral norm.
Let $A$ be an $m\times m$ real matrix of rank $r$, then we seek rank-$n$ matrices $R_n$, $n<r$, such that
\begin{align*}
\|A-R_n\|\leq \|A-B\|,
\end{align*}
for any $m\times m$ matrix $B$ of rank $n$, and where $\|\cdot\|$ is the operator norm induced by the Euclidean norm, i.e., the spectral norm.

The singular value decomposition (SVD) is an essential  
tool for analyzing and solving the best low-rank approximation problem; see, e.g., \cite[Chapter~3]{Blum:2020}.
Let $A=U \Sigma V^T$ be any SVD of $A$, i.e., $\Sigma$ is the diagonal matrix whose diagonal entries, 
$$
\sigma_1\geq \sigma_2\geq\dots\geq \sigma_r>\sigma_{r+1}=\dots= \sigma_m=0,
$$ 
are the singular values of $A$, and $U$ and $V$ are orthonormal matrices. We further let $\bfu_j$ and $\bfv_j$ denote the $j$-th column vector of $U$ and $V$. If $n<r$, then the Eckhart--Young theorem \cite[Theorem~2.4.8]{Golub:2013} states that the rank-$n$ matrix
\begin{equation}\label{eq:ba-SVD}
R_n=\sum_{i=1}^n\sigma_i\bfu_i\bfv_i^T
\end{equation}
satisfies
\begin{equation}\label{eq:rank-n-approx-gen}
\|A-R_n\|=\min_{\operatorname{rank}(B)=n} \|A-B\|=\sigma_{n+1},
\end{equation}
and is thus a best rank-$n$ approximation to $A$ in the spectral norm. However, in many applications one is interested in finding low-rank approximations that preserve certain \emph{structures} in the original matrix $A$, i.e., structured low-rank approximation \cite{Chu:2003,Higham:1989,Markovsky:2008,Ottaviani:2014, Ishteva:2014,Grussler:2018}. Preserving these structures could exclude the matrix $R_n$ in \cref{eq:ba-SVD} from being a suitable approximation, and in general one looks for \emph{near-best} approximations that preserve these structures. In this paper we provide a classification of other best low-rank approximations to $A$ than $R_n$ in \cref{eq:ba-SVD}. One could then search among these matrices for best low-rank approximations that have the desired structure or other problem-oriented properties. 
In fact, the special case of best rank-$1$ approximations to Hankel matrices has already been considered in \cite{Antoulas:1997}; see also \cite{Knirsch:preprint} where further results and efficient algorithms for structured best rank-$1$ approximations to Hankel matrices can be found. We also remark that the problem of finding best low-rank approximations in other (entry-wise) matrix norms has been studied in \cite{Pinkus:2012} and \cite{Georgieva:2017}.

Observe that the matrix $R_n$ in \cref{eq:ba-SVD} is clearly not unique if $\sigma_n=\sigma_{n+1}>0$ and it is then straightforward to find other best rank-$n$ approximations to $A$. 
If
 $\sigma_n>\sigma_{n+1}>0$
it is known that the matrix in \cref{eq:ba-SVD} is the unique best rank-$n$ approximation to $A$ in the Frobenius norm; see, e.g., \cite[Section~7.4.2]{Horn:2013}.
However, as argued by Tropp \cite[p.~122]{Tropp:2015}, error bounds in the Frobenius norm are not always useful in cases of practical interest and can even be completely ``vacuous''; see also \cite{Li:2017,Musco:2015} for a similar argument. It is therefore more desirable to look for low-rank approximations in the spectral norm. For this norm
the problem has infinitely many solutions whenever $\sigma_{n+1}>0$, because any matrix of the form
\begin{equation}
\label{eq:no-unique-spectral}
\sum_{i=1}^n(\sigma_i+\epsilon_i)\bfu_i\bfv_i^T, \quad -\sigma_{n+1}\leq\epsilon_i\leq\sigma_{n+1},
\end{equation}
solves \cref{eq:rank-n-approx-gen}.
In this paper we look for more general solutions of the form $\sum_{i=1}^n\bfx_i\bfy_i^T$ with $\bfx_i,$ $\bfy_i\in \mathbb{R}^m$, other than \cref{eq:ba-SVD} and \cref{eq:no-unique-spectral}, to the best low-rank approximation problem in \cref{eq:rank-n-approx-gen}. 

Our approach to finding other best rank-$n$ approximations to $A$ consists of two steps: first we relate this problem to Kolmogorov $n$-widths \cite{Kolmogorov:36} and then we solve the $n$-width problem.
The Kolmogorov $n$-width of a set in a normed linear space is the minimal distance to the given set from all possible $n$-dimensional subspaces. An $n$-dimensional (sub)space is \emph{optimal} when it realizes this minimal distance.
We provide a classification of all the optimal $n$-dimensional spaces for the image of the Euclidean unit ball under $A$, which can be recognized as an $r$-dimensional ellipsoid in $\mathbb{R}^m$. It turns out that the corresponding Kolmogorov $n$-width equals $\sigma_{n+1}$ and that any orthonormal basis in such $n$-dimensional optimal space generates a best rank-$n$ approximation to $A$. This results in a large variety of best rank-$n$ approximations beyond the truncated SVD solution in \cref{eq:ba-SVD}, and can be exploited to obtain low-rank approximations with problem-oriented properties.

As a byproduct of our results we classify all $n$-dimensional spaces that achieve the minimum in the following min-max formulation for the singular values of~$A$:
\begin{equation}\label{eq:min-max}
\sigma_{n+1} = \min_{\XX_n}\max_{\bfz\perp \XX_n} \sqrt{\frac{\bfz^TAA^T\bfz}{\bfz^T\bfz}}.
\end{equation}
This formula is a direct consequence of the Courant--Fischer theorem \cite[Section~7.3]{Horn:2013}.
It is easily verified that $\XX_n=\spn\{\bfu_1,\ldots,\bfu_n\}$ achieves the minimum in \cref{eq:min-max}. However, as already pointed out in \cite{Karlovitz:73,Karlovitz:76}, this space is unique only in very special cases.

For further relations between the $n$-width and matrix theory we refer the reader to the survey paper \cite{Pinkus:79}, and for further $n$-width results in general to the book \cite{Pinkus:85}.

In this paper we restrict our attention to the case where the $(n+1)$-st singular value is non-zero and unique, i.e.,
\begin{equation}
\label{eq:sigmas-unique-weak}
\sigma_n>\sigma_{n+1}>\sigma_{n+2}\geq0.
\end{equation}
Besides the above discussion, this assumption is taken to simplify the exposition since it ensures that the $(n+1)$-st left singular vector of $A$ is unique (up to multiplication by constants). All our findings can be easily extended to rectangular matrices $A$ of rank $r$.

The remainder of this paper is organized as follows. \Cref{sec:low-rank} states the definitions of Kolmogorov $n$-widths and optimal spaces for the image of the Euclidean unit ball by $A$ and connects them with best rank-$n$ approximations to $A$. Some known necessary or sufficient conditions for a subspace to be optimal are recalled in \cref{sec:optimal-subspaces}. \Cref{sec:optimality-crit} is the core of the paper and provides characterizations of optimal subspaces by means of some optimality criteria. We discuss them in detail for the important case of best rank-$1$ approximation in \cref{sec:1-width}. Some alternative optimality criteria are collected in \cref{sec:optimality-crit-alt}. \Cref{sec:tp-matrices,sec:seq-subspaces}, inspired by similar results for integral operators in $L^2$, present a simple explicit construction to obtain a sequence of optimal $n$-dimensional subspaces once an initial optimal subspace is given. This construction can be exploited to obtain alternative best rank-$n$ spectral approximations for any matrix $A$. Some concluding remarks are collected in \cref{sec:conclusion}.

\section{Kolmogorov $n$-widths and rank-$n$ approximations}\label{sec:low-rank}
Let $A$ be an $m \times m$ real matrix of rank $r$, and define the subset of $\RR^m$,
\begin{align*}
\calA &:= \{A\bfx : \bfx \in \RR^m,\, \|\bfx\| \le 1\}, 
\end{align*}
where $\| \cdot \|$ is the Euclidean norm in $\RR^m$.
Note that $\calA$ can be recognized as a (filled)
$r$-dimensional ellipsoid in $\mathbb{R}^m$, 
where the line segments $[-\sigma_i \bfu_i, \sigma_i \bfu_i]$, $i=1,\ldots,r$, are its principal axes.
The spectral norm of $A$ is the induced operator norm given by
\begin{align*}
\|A\|:=\max_{\|\bfx\|\leq1} \|A\bfx\|,
\end{align*}
and it can be shown that $\|A\|=\|A^T\|=\sigma_1$.
For an $n$-dimensional subspace $\XX_n$ of $\RR^m$, where $0 \le n \le m$,
we define the distance to $\calA$ from $\XX_n$ by
\begin{equation}\label{eq:E}
 E(\calA, \XX_n) :=\max_{\bfa \in \calA} \dist(\bfa,\XX_n)= \max_{\bfa \in \calA} \min_{\bfx \in \XX_n} \|\bfa-\bfx\|. 
\end{equation}
Then, the Kolmogorov $n$-width of $\calA$, relative to the Euclidean norm in $\RR^m$, is defined by
$$ d_n(\calA) := \min_{\XX_n} E(\calA, \XX_n). $$
A subspace $\XX_n$ of $\RR^m$ is called an optimal subspace for $\calA$
provided that
$$E(\calA, \XX_n) = d_n(\calA). $$
Here the $0$-dimensional subspace $\XX_0$ of $\RR^m$ is $\{0\}$. 

We can determine the $n$-width of $\calA$ for any $n=0,\ldots,m$ as follows. Let $P_n$ be the orthogonal projection onto $\XX_n$. Then,
\begin{equation}\label{eq:max0}
\begin{aligned}
E(\calA, \XX_n) &=  \max_{\bfa \in \calA} \|\bfa-P_n\bfa\| = \max_{\|\bfx\| \le 1} \|(I-P_n)A\bfx\|=\|(I-P_n)A\|\\
&=\|A^T(I-P_n)\|
=\max_{\bfx\neq 0} \frac{\|A^T(I-P_n)\bfx\|}{\|\bfx\|},
\end{aligned}
\end{equation}
where we have used that the spectral norm of a matrix equals the spectral norm of its adjoint. By letting $\bfx=\bfy\oplus\bfz$ for $\bfy\in\XX_n$ and $\bfz\perp\XX_n$ one can check that the last maximum in \eqref{eq:max0} is achieved for $\bfy=0$. This implies that
\begin{equation}\label{eq:max}
 E(\calA, \XX_n) = \max_{\bfz\perp \XX_n} \frac{\|A^T\bfz\|}{\|\bfz\|}
 = \max_{\bfz\perp \XX_n} \sqrt{\frac{\bfz^TAA^T\bfz}{\bfz^T\bfz}}.
\end{equation}
Now, using the definition of $d_n(\calA)$, together with \cref{eq:min-max} and \cref{eq:max}, we observe that
\begin{equation}\label{eq:width=sv}
d_n(\calA) = \sigma_{n+1}, \quad n=0,1,\ldots, m-1.
\end{equation}
We also note that it easily follows
from the definition of the $n$-width that $d_m(\calA) = 0$,
due to the fact that the only choice of a subspace of $\RR^m$ of
dimension $m$ is $\XX_m = \RR^m$.
Thus, we have
$$ (d_0(\calA),d_1(\calA),\ldots, d_m(\calA))
   = (\sigma_1,\sigma_2,\ldots,\sigma_m,0), $$
and, as mentioned in the introduction, $\XX_n=\spn\{\bfu_1,\ldots,\bfu_n\}$ is an optimal space for $\calA$.

The relation between Kolmogorov $n$-widths and rank-$n$ approximations is contained in the next two theorems. 

\begin{theorem}\label{thm:low-rank-1} 
Assume that the vectors $\bfx_i$, $i=1,\ldots,n$, are orthonormal, and define $\bfy_i:=A^T\bfx_i$, $i=1,\ldots,n$. If $\XX_n:=\spn\{\bfx_1,\ldots,\bfx_n\}$, then
\begin{equation*}
\|A-\sum_{i=1}^n\bfx_i\bfy_i^T\| = E(\calA,\XX_n),
\end{equation*}
and, consequently, the matrix $\sum_{i=1}^n\bfx_i\bfy_i^T$ is a best rank-$n$ approximation to $A$ if and only if the subspace $\XX_n$ is optimal for $\calA$.
\end{theorem}
\begin{proof}
Let $P_n$ be the orthogonal projection onto $\XX_n$. It follows from \cref{eq:max0} that
\begin{align*}
E(\calA,\XX_n) &=\max_{\|\bfx\|= 1} \|A\bfx- P_nA\bfx\|
=\max_{\|\bfx\|= 1} \|A\bfx- \sum_{i=1}^n(\bfx_i^TA\bfx)\bfx_i\| 
\\
&=\max_{\|\bfx\|= 1} \|A\bfx- \sum_{i=1}^n((A^T\bfx_i)^T\bfx)\bfx_i\| 
= \max_{\|\bfx\|= 1} \|A\bfx- \sum_{i=1}^n\bfx_i\bfy_i^T\bfx\|
\\
&=\|A-\sum_{i=1}^n\bfx_i\bfy_i^T\|.
\end{align*}
Since $d_n(\calA)=\sigma_{n+1}$, the result follows.
\end{proof}

We remark that the above theorem can be considered as an extension of an observation in \cite{Pinkus:2012}. 
Define the subset $\calA_T:= \{A^T\bfx : \bfx \in \RR^m,\, \|\bfx\| \le 1\}$ and observe that $d_n(\calA_T)=\sigma_{n+1}$ since $A^T$ has the same singular values as $A$.
The following result can be obtained by a similar argument as for \cref{thm:low-rank-1}.
\begin{theorem}\label{thm:low-rank-2}
Assume that the vectors $\bfy_i$, $i=1,\ldots,n$, are orthonormal, and define $\bfx_i:=A\bfy_i$, $i=1,\ldots,n$.
If $\YY_n:=\spn\{\bfy_1,\ldots,\bfy_n\}$, then
\begin{equation*}
\|A-\sum_{i=1}^n\bfx_i\bfy_i^T\| = E(\calA_T,\YY_n),
\end{equation*}
and, consequently, the matrix $\sum_{i=1}^n\bfx_i\bfy_i^T$ is a best rank-$n$ approximation to $A$ if and only if the subspace $\YY_n$ is optimal for $\calA_T$.
\end{theorem}

We remark that if $\XX_n$ is an optimal subspace for $\calA$ then it follows from the results of \cref{sec:seq-subspaces} that $A^T(\XX_n)=\spn\{A^T\bfx_1,\ldots,A^T\bfx_n\}$ is an optimal space for $\calA_T$. Thus, the $\bfy_i$, $i=1,\ldots,n$, in \cref{thm:low-rank-1} span an optimal space for $\calA_T$ whenever $\XX_n$ is optimal for $\calA$. A similar observation holds for \cref{thm:low-rank-2} and we refer the reader to \cref{sec:seq-subspaces} for the details.

The classical truncated SVD approximation to $A$ can be recovered by taking either $\bfx_i=\bfu_i$, $i=1,\ldots,n$, in \cref{thm:low-rank-1} or $\bfy_i=\bfv_i$, $i=1,\ldots,n$, in \cref{thm:low-rank-2}.
From the above theorems we observe that a classification of all the optimal spaces for $\calA$ and $\calA_T$ leads to a classification of several best low-rank approximations to $A$. Such a classification is the goal of the remainder of this paper.

Equivalence between best rank-$n$ approximation and optimality of the corresponding subspaces for the Kolmogorov $n$-width has been shown under the assumptions of either \cref{thm:low-rank-1} or \cref{thm:low-rank-2} (see also \cref{pro:n=1_symmetry}). It is an open question whether this equivalence holds more generally.

\section{Optimal subspaces}\label{sec:optimal-subspaces}
Let us start searching for optimal subspaces for $\calA$. 
From now on we assume that the singular values of $A=U \Sigma V^T$ satisfy \cref{eq:sigmas-unique-weak}.
Here we recall some optimality conditions from Karlovitz \cite{Karlovitz:76}. The following condition is necessary for the optimality of a subspace; see \cite[Theorem~1]{Karlovitz:76} for a proof.
\begin{theorem}\label{thm:Karlovitz-nec}
Given $n<r$, if  $\XX_n$ is an optimal subspace for $\calA$, then $\XX_n \perp \bfu_{n+1}$.
\end{theorem}

As mentioned in the introduction, under the assumption \cref{eq:sigmas-unique-weak} the left singular vector $\bfu_{n+1}$ is unique (up to multiplication by constants).
In general, if there are multiple equal singular values for $A$, then an optimal subspace $\XX_n$ must be orthogonal to a certain subspace spanned by the left singular vectors of $A$; see \cite[Theorem~1]{Karlovitz:76} for the details.

Note that in the special case $r=m$ and $n=m-1$, \cref{thm:Karlovitz-nec} implies the uniqueness of the optimality of
\begin{equation}\label{eq:optimal:N-1}
\XX_{m-1}=\spn\{\bfu_1,\ldots,\bfu_{m-1}\}.
\end{equation}

In addition to the necessary condition in \cref{thm:Karlovitz-nec}, Karlovitz also proved a sufficient condition for optimality. Roughly speaking, it states that any subspace ``sufficiently close'' to the optimal space $\spn\{\bfu_1,\ldots\bfu_n\}$ must be optimal whenever it satisfies the necessary condition of \cref{thm:Karlovitz-nec}. The precise condition is stated in the following theorem; see \cite[Theorem~1]{Karlovitz:76} for a proof.
\begin{theorem}\label{thm:Karlovitz-suff}
Given
$n < \min\{m-1,r\}$,
if
 $\XX_n\perp \bfu_{n+1}$ and 
\begin{equation}\label{ineq:Karlovitz}
\sum_{i=1}^n\|\bfu_i-P_n\bfu_i\|^2\sigma^2_i\leq \sigma^2_{n+1}-\sigma^2_{n+2},
\end{equation}
where $P_n$ is the orthogonal projection onto $\XX_n$, then $\XX_n$ is an optimal subspace for $\calA$. 
\end{theorem}

\section{Optimality criteria}\label{sec:optimality-crit}

With the aim of deriving novel conditions for optimality of subspaces, we first provide a characterization of the distance $E(\calA,\XX_n)$.
\begin{lemma}\label{lem:Eisaneigenvalue}
Let $P_n$ be the orthogonal projection onto $\XX_n$. The distance $E(\calA,\XX_n)$ is equal to the square root of the largest eigenvalue
of
\begin{equation}\label{eq:B}
\Sigma^2 - \Sigma U^T P_n U\Sigma.
\end{equation}
\end{lemma}
\begin{proof}
First note that $P_n = P_n^2 = P_n^T P_n$. 
Similar to \cite[Theorem~2.3]{Melkman:78}, by using \cref{eq:E} and the definition of $\calA$ we deduce that
\begin{align*}
E(\calA,\XX_n)^2 &= \max_{\|\bfx\|\leq 1}\|A\bfx-P_nA\bfx\|^2 =  \max_{\bfx\neq 0}\frac{((I-P_n)A\bfx,(I-P_n)A\bfx)}{(\bfx,\bfx)}
\\
&=\max_{\bfx\neq 0}\frac{(A^T(I-P_n)A\bfx,\bfx)}{(\bfx,\bfx)},
\end{align*}
and so $E(\calA,\XX_n)$ is the square root of the largest eigenvalue
of $M:=A^T(I-P_n)A$. From the SVD of $A$ we see
that $M = VBV^T$, where $B:=\Sigma^2 - \Sigma U^T P_n U\Sigma$ is the matrix in \cref{eq:B}.
Since $B$ is a similarity transformation of $M$, they share the same
eigenvalues.
\end{proof}

The characterization of $E(\calA,\XX_n)$ in \cref{lem:Eisaneigenvalue} forms the basis for our optimality criteria.
Let
\begin{equation*}
C_{n+1} :=
\sigma_{n+1}^2 I - \Sigma^2 + \Sigma U^T P_n U \Sigma,
\end{equation*}
 and let
$C_{n+1}[i_1,\ldots,i_k]$ denote the $k \times k$ submatrix of $C_{n+1}$
consisting of the rows and columns with indices $i_1,\ldots,i_k$.

\begin{lemma}\label{lem:optimality_criterion}
The subspace
$\XX_n$ is optimal for $\calA$ if and only if
$\XX_n \perp \bfu_{n+1}$ and
$C_{n+1}[1,\ldots,n,n+2,\ldots,m]$ is positive semi-definite.
\end{lemma}
\begin{proof}
Suppose $\XX_n$ is optimal for $\calA$. Then, from \cref{eq:width=sv} we deduce that $E(\calA,\XX_n)= \sigma_{n+1}$, and by \cref{lem:Eisaneigenvalue} we have that $C_{n+1}$ is positive semi-definite. Conversely, if $C_{n+1}$ is positive semi-definite, then using again the same lemma we can conclude that $\XX_n$ is optimal for $\calA$.
Moreover, by \cref{thm:Karlovitz-nec},
$\XX_n \perp \bfu_{n+1}$ 
and the $(n+1)$-st row and $(n+1)$-st column of
$C_{n+1}$ are zero, and so
$C_{n+1}$ is positive semi-definite if and only if
$C_{n+1}[1,\ldots,n,n+2,\ldots,m]$ is positive semi-definite.
\end{proof}

\begin{proposition}\label{prop:optimality_criterion2}
The subspace
$\XX_n$ is optimal for $\calA$ if and only if 
$\XX_n \perp \bfu_{n+1}$ and
\begin{equation}\label{eq:mycond}
  \det(C_{n+1}[\calJ]) \ge 0,
\end{equation}
for any set of indices $\calJ\subseteq \{1,\ldots,n,n+2,\ldots,m\}$ such that
$\{1,\ldots,n\} \cap \calJ \ne \emptyset$.
\end{proposition}
\begin{proof}
By the previous lemma,
$\XX_n$ is optimal for $\calA$ if and only if
$\XX_n \perp \bfu_{n+1}$ and the matrix
$C_{n+1}[1,\ldots,n,n+2,\ldots,m]$ is positive semi-definite.
The latter is equivalent to the two conditions
\begin{equation}\label{eq:condother}
  \det(C_{n+1}[\calJ]) \ge 0, \quad \calJ\subseteq \{n+2,\ldots,m\},
\end{equation}
and \cref{eq:mycond}.
Thus, to complete the proof it is sufficient to
show that \cref{eq:condother} holds for all $\XX_n$, i.e., that
$C_{n+1}[n+2,n+3,\ldots,m]$ is positive semi-definite for any $\XX_n$.
To see this, let $\bfx = [0,\ldots,0,x_{n+2},\ldots,x_m]^T \in \RR^m$.
Then, noting that $P_n = P_n^2 = P_n^T P_n$,
\begin{equation}\label{eq:posdef}
 \bfx^T C_{n+1} \bfx
  = \sum_{i=n+2}^m (\sigma_{n+1}^2 - \sigma_i^2) x_i^2
    + \|P_n U \Sigma \bfx \|^2 \ge 0,
\end{equation}
and thus $C_{n+1}[n+2,n+3,\ldots,m]$ is indeed positive semi-definite.
\end{proof}

Alternatively, we can consider a sufficient condition for
optimality that involves checking the sign of only $n$ determinants.

\begin{corollary}\label{thm:optimality_sufficiency}
The subspace
$\XX_n$ is optimal for $\calA$ if
$\XX_n \perp \bfu_{n+1}$ and
\begin{equation}\label{eq:mycondstrict}
  \det(C_{n+1}[k,k+1,\ldots,n,n+2,\ldots,m]) > 0,
  \quad k=1,2,\ldots,n.
\end{equation}
\end{corollary}
\begin{proof}
The subspace $\XX_n$ is optimal for $\calA$ if $\XX_n \perp \bfu_{n+1}$ and the matrix \linebreak
$C_{n+1}[1,\ldots,n,n+2,\ldots,m]$ is positive definite.
The latter is equivalent to the two conditions
\begin{equation}\label{eq:condotherstrict}
  \det(C_{n+1}[k,k+1,\ldots,m]) > 0, \quad k= n+2,\ldots,m,
\end{equation}
and \cref{eq:mycondstrict}.
But \cref{eq:condotherstrict} holds for any $\XX_n$
since inequality \cref{eq:posdef} is strict unless $x_{n+2} = x_{n+3} = \cdots = x_m = 0$.
\end{proof}

Let us now express the subspace $\XX_n$ in the form
\begin{equation}\label{eq:X_W}
\XX_n=\spn\{\bfx_1,\ldots,\bfx_n\},
\end{equation}
where $\bfx_1,\ldots,\bfx_n$ are orthonormal vectors in $\RR^m$.
Then, the projection $P_n$ equals $XX^T$ where $X\in\RR^{m,n}$ is the matrix
whose columns are $\bfx_1,\ldots,\bfx_n$.
We can express these vectors in the basis $\bfu_1,\ldots,\bfu_m$,
and write
\begin{equation*}
 \bfx_j = \sum_{i=1}^m w_{ij} \bfu_i, \quad j=1,\ldots,n,
\end{equation*}
for coefficients $w_{ij} \in \RR$.
Letting $W\in\RR^{m,n}$ be the matrix
$[w_{ij}]_{i=1,\ldots,m,j=1,\ldots,n}$, we find that
$$ X = U W, $$
and it follows that
$$ P_n = U W W^T U^T, $$
and therefore, that
\begin{equation}\label{eq:C_W}
C_{n+1} = \sigma_{n+1}^2 I - \Sigma^2 + \Sigma W W^T \Sigma.
\end{equation}
Note that $W = U^T X$, which implies
$$ W^T W = X^T X = I, $$
and so the columns $\bfw_1,\ldots,\bfw_n$ of $W$ are orthonormal.

We can then further sharpen the condition of
\cref{prop:optimality_criterion2}
by making use of a matrix determinant identity.

\begin{lemma}\label{lem:det-formula}
Suppose $\XX_n$ is as in \cref{eq:X_W} and $C_{n+1}$ as in \cref{eq:C_W}.
If $\calJ\subseteq \{1,\ldots,n,n+2,\ldots,m\}$ is any set of indices, then
\begin{align*}
\det(C_{n+1}[\calJ]) =  \det(M_\calJ)\prod_{k\in \calJ}(\sigma_{n+1}^2-\sigma_k^2),
\end{align*}
where $M_\calJ = [m_{ij}]_{i,j=1,\ldots,n}$ has the elements
\begin{align}\label{eq:M-mat}
m_{ij} = 
 \sigma_{n+1}^2\sum_{k\in \calJ}\frac{w_{ki} w_{kj}}
      {\sigma_{n+1}^2-\sigma_k^2}+
   \sum_{k\not\in \calJ}w_{ki} w_{kj}.
\end{align}
\end{lemma}
\begin{proof}
Let $\calI_k:=\{1,\ldots,k\}$, and let $W[\calJ,\calI_n]$ be the submatrix of $W$ consisting of the rows with indices in $\calJ$ and columns with indices in $\calI_n$.
We use the fact that for any non-singular matrix
$F \in \RR^{m,m}$ and any matrix $G \in \RR^{m,n}$, it holds that
\begin{equation*}
\det(F+GG^T)=\det(I+G^T F^{-1} G)\det(F);
\end{equation*}
see~\cite[Theorem~18.1.1]{Harville:1997}.
Applying this identity with
\begin{equation*}
F:= (\sigma_{n+1}^2 I - \Sigma^2)[\calJ],\quad
G:= \Sigma[\calJ]W[\calJ,\calI_n],
\end{equation*}
we find that
$$
\det(C_{n+1}[\calJ]) =\det(I+ W[\calJ,\calI_{n}]^TD W[\calJ,\calI_{n}])\det(F),
$$
where $D:=\Sigma[\calJ] F^{-1}\Sigma[\calJ]$
is the diagonal matrix given by
\begin{align*}
D_{kk} = \frac{\sigma_k^2}{\sigma_{n+1}^2-\sigma_k^2}, \quad k\in \calJ.
\end{align*}
Moreover, we find that 
\begin{align*}
I &= \left[\sum_{k\in \calI_m}w_{ki} w_{kj}\right]_{i,j=1,\ldots,n},
\\
W[\calJ,\calI_{n}]^TD W[\calJ,\calI_{n}] &= \left[ 
    \sum_{k\in \calJ}\frac{\sigma_k^2 w_{ki} w_{kj}}
      {\sigma_{n+1}^2-\sigma_k^2}\right]_{i,j=1,\ldots,n},
\end{align*}
and therefore $M_\calJ=I+ W[\calJ,\calI_{n}]^TD W[\calJ,\calI_{n}]$ since
$$ m_{ij} = 
     \sum_{k\in \calJ} w_{ki} w_{kj}
    + 
     \sum_{k\not\in \calJ} w_{ki} w_{kj}
    + \sum_{k\in \calJ}\frac{\sigma_k^2 w_{ki} w_{kj}}
      {\sigma_{n+1}^2-\sigma_k^2}.
$$
Finally, since $F$ is diagonal, we have
$$\det(F)= \prod_{k\in \calJ}(\sigma_{n+1}^2-\sigma_k^2),$$
and the result follows.
\end{proof}

\begin{theorem}\label{thm:optimality_criterion3}
The subspace $\XX_n$ is optimal for ${\calA}$ if and only if
$\XX_n \perp \bfu_{n+1}$ and for all sets of indices $\calJ\subseteq \{1,\ldots,n,n+2,\ldots,m\}$ such that $\{1,\ldots,n\} \cap \calJ \ne \emptyset$ we have
\begin{equation*}
  (-1)^s \det(M_\calJ) \geq 0,
\end{equation*}
 where $s$ is the cardinality of $\{1,\ldots,n\} \cap \calJ$ and $M_\calJ$ is the matrix given in \cref{eq:M-mat}.
\end{theorem}
\begin{proof}
From \cref{prop:optimality_criterion2} we know that $\XX_n$ is optimal for ${\calA}$ if and only if
$\XX_n \perp \bfu_{n+1}$ and for all sets of indices $\calJ\subseteq \{1,\ldots,n,n+2,\ldots,m\}$ such that $\{1,\ldots,n\} \cap \calJ \ne \emptyset$, we have
\begin{equation*}
  \det(C_{n+1}[J]) = \det(M_\calJ)\prod_{k\in \calJ}(\sigma_{n+1}^2-\sigma_k^2) \ge 0.
\end{equation*}
Now, since the singular values satisfy \cref{eq:sigmas-unique-weak}
we find that 
\begin{align*}
(-1)^s \prod_{k\in \calJ}(\sigma_{n+1}^2-\sigma_k^2) > 0,
\end{align*}
which gives the result.
\end{proof}

There is a freedom in the choice of the basis $\bfx_1,\ldots,\bfx_n$ for the space $\XX_n$ in \cref{eq:X_W}, and this freedom will affect the matrices in the above optimality criterion. Looking at the sufficient condition in \cref{thm:Karlovitz-suff}, a natural candidate for a basis of $\XX_n$ seems to be $P_n\bfu_1,\ldots,P_n\bfu_n$, as long as they are linearly independent. If they are, then they can be orthonormalized by a Gram--Schmidt process before being used in \cref{thm:optimality_criterion3}. Let us now prove that $P_n\bfu_1,\ldots,P_n\bfu_n$ are in fact linearly independent whenever $\XX_n$ is optimal.

\begin{proposition}
Let $P_n$ be the orthogonal projection onto $\XX_n$.
If $\XX_n$ is optimal for $\calA$, then
$P_n \bfu_1,\ldots,P_n \bfu_n$ are linearly independent.
\end{proposition}
\begin{proof}
Suppose, on the contrary, that there are coefficients
$c_1,\ldots,c_n \in \RR$, not all zero, such that
$$ \sum_{i=1}^n c_i P_n \bfu_i = 0. $$
Then,
$$ P_n \left(\sum_{i=1}^n c_i \bfu_i\right) = 0, $$
which we can write as
$$ P_n U \bfc = 0, $$
where $\bfc = [c_1,\ldots,c_n,0,\ldots,0]^T \in \RR^m$.
Let $\bfy\in \mathbb{R}^m$ be such that $ \Sigma\bfy= \bfc$.
Then,
$$ P_n U \Sigma \bfy = 0, $$
and therefore,
$$ \bfy^T (\Sigma^2 - \Sigma U^T P_n U \Sigma) \bfy = \bfy^T \Sigma^2 \bfy. $$
Since not all the coefficients $c_1,\ldots,c_n$ are zero,
not all the coefficients $y_1,\ldots,y_n$ are zero.
Therefore, we can form the Rayleigh quotient of $B:=\Sigma^2 - \Sigma U^T P_n U \Sigma$ and $\bfy$,
and we find
$$ \frac{\bfy^T (\Sigma^2 - \Sigma U^T P_n U \Sigma) \bfy}{\bfy^T \bfy}
   = \frac{\bfy^T \Sigma^2 \bfy}{\bfy^T \bfy}
   = \frac{\sum_{i=1}^n y_i^2 \sigma_i^2}
          {\sum_{i=1}^n y_i^2}
     \ge \sigma_n^2, $$
and so $E(\calA,\XX_n) \ge \sigma_n$ and $\XX_n$ is not optimal for ${\calA}$
(see \cref{lem:Eisaneigenvalue}).
\end{proof}

\section{Optimality for the $1$-width and best rank-$1$ approximation} \label{sec:1-width}
For the $1$-width we can derive an explicit form of the
optimality criterion in \cref{thm:optimality_criterion3}.
Suppose $\XX_1 = \spn\{\bfx_1\}$
for some $\bfx_1 = \sum_{i=1}^mw_i\bfu_i \in \RR^m$ with $\|\bfx_1\| = 1$.

\begin{theorem}\label{thm:optimality_criterion_n=1}
The subspace $\XX_1$ is optimal for ${\calA}$ if and only if
$w_2 = 0$ and
\begin{equation}\label{ineq:optimality_criterion_n=1}
\sum_{i=3}^m \frac{w_i^2}{\sigma_2^2-\sigma_i^2}\leq \frac{w_1^2}{\sigma_1^2-\sigma_2^2}.
\end{equation}
\end{theorem}
\begin{proof}
Note that for $n=1$ the matrix $M_\calJ$ in \cref{eq:M-mat} is a scalar. Using \cref{thm:optimality_criterion3}, 
the subspace $\XX_1$ is optimal if and only if $w_2 = 0$ and
$$M_\calJ =
  \sigma_2^2\sum_{i\in \calJ}\frac{w_i^2}{\sigma_2^2-\sigma_i^2}
  +  \sum_{i\not\in \calJ}w_i^2\leq 0, $$
for any subset $\calJ$ of $\{1,3,\ldots,m\}$ that contains $1$.
Since $w_2 = 0$, this is equivalent to
\begin{align}\label{ineq:optimality_criterion_J}
\sigma_2^2\sum_{i\in \calJ}\frac{w_i^2}{\sigma_2^2-\sigma_i^2}+
   \sum_{i\in \calK}w_i^2\leq 0,
\end{align}
where $\calK = \{1,3,\ldots,m\} \setminus J$.
Now, if $\calJ=\{1,3,\ldots,m\}$, then $\calK = \emptyset$
and \cref{ineq:optimality_criterion_J} is equivalent to
\cref{ineq:optimality_criterion_n=1}.
If, on the other hand, $\calJ$ is a strict subset of $\{1,3,\ldots,m\}$, then
\begin{align*}
\sigma_2^2\sum_{i\in \calJ}\frac{w_i^2}{\sigma_2^2-\sigma_i^2}+
 \sum_{i\in \calK}w_i^2\leq 
  \sigma_2^2\sum_{i\in \calJ}\frac{w_i^2}{\sigma_2^2-\sigma_i^2}+
  \sum_{i\in \calK}\frac{\sigma_2^2}{\sigma_2^2-\sigma_i^2} w_i^2 =
 \sigma_2^2\sum_{\substack{i=1\\ i\neq 2}}^m\frac{w_i^2}{\sigma_2^2-\sigma_i^2}
  \leq 0,
\end{align*}
since $\sigma_2\geq \sigma_2-\sigma_j$ for any $j \in \{3,\ldots,m\}$.
This concludes the proof.
\end{proof}

Observe that by combining the above result with either \cref{thm:low-rank-1} or \cref{thm:low-rank-2} we obtain a characterization of several best rank-$1$ approximations to $A$.
We remark that a condition similar to \cref{ineq:optimality_criterion_n=1} was found by Antoulas \cite{Antoulas:1997} in the special case of rank-1 approximation to Hankel matrices.

The optimality criterion in \cref{ineq:optimality_criterion_n=1} is trivially satisfied by the classical optimal space $\spn{\{\bfu_1\}}$ and it provides a characterization of ``how far'' a one-dimensional space can deviate from $\spn{\{\bfu_1\}}$  and still remain optimal.
Specifically, let $\bfx_1 = \sum_{i=1}^mw_i\bfu_i \in \RR^m$ with $\|\bfx_1\| = 1$, then \cref{thm:optimality_criterion_n=1} shows that if $\XX_1= \spn\{\bfx_1\}$ is optimal for $\calA$ and $A\neq 0$, then $w_1\neq 0$. Indeed, if $w_1=0$, then from \cref{ineq:optimality_criterion_n=1} we have that $w_i=0$, $i=2,\ldots,m$ and so $\bfx_1=0$.
The space $\XX_1=\{0\}$ can only be optimal for the $1$-width of $\calA$ if $\sigma_1=\sigma_2$, which contradicts assumption \cref{eq:sigmas-unique-weak}.

Let us now compare the result in \cref{thm:optimality_criterion_n=1} with the sufficient condition of Karlovitz (\cref{thm:Karlovitz-suff}).
Note that \cref{ineq:optimality_criterion_n=1} is equivalent to 
\begin{equation}\label{ineq:optimality_criterion2}
\sum_{i=3}^m \frac{\sigma_1^2-\sigma_i^2}{\sigma_2^2-\sigma_i^2}w_i^2\leq 1,
\end{equation}
by using $w_1^2=1-\sum_{i=3}^mw_i^2$ and $w_2=0$. On the other hand, for $n=1$, the left-hand side of \cref{ineq:Karlovitz} equals
\begin{equation*}
\|\bfu_1-(\bfu_1,\bfx_1)\bfx_1\|^2\sigma_1^2 = (\|\bfu_1\|^2-(\bfu_1,\bfx_1)^2)\sigma_1^2 = (1-w_1^2)\sigma_1^2= \sum_{i=3}^mw_i^2\sigma_1^2,
\end{equation*}
and so, condition \cref{ineq:Karlovitz} is equivalent to
\begin{equation}\label{ineq:Karlovitz_n=1}
\sum_{i=3}^m\frac{\sigma_1^2}{\sigma_2^2-\sigma_3^2}w_i^2\leq 1.
\end{equation}
Since the singular values are decreasing, we have
\begin{equation}\label{ineq:comparison_Karlovitz}
\frac{\sigma_1^2-\sigma_i^2}{\sigma_2^2-\sigma_i^2}\leq \frac{\sigma_1^2}{\sigma_2^2-\sigma_3^2}, \quad i=3,\ldots,m,
\end{equation}
and condition \cref{ineq:Karlovitz_n=1} implies \cref{ineq:optimality_criterion2}, as expected. However, we note that the case $i=3$ in \cref{ineq:comparison_Karlovitz} is a strict inequality if $\sigma_3> 0$. Thus, for $n=1$, the sufficient condition in \cref{thm:Karlovitz-suff} is stronger than necessary whenever $\sigma_3>0$.

\begin{example}\label{ex:n=1_3D}
Let $m=3$ and consider the space $\XX_1 = \spn\{\bfx_1\}$ for some $\bfx_1 = w_1\bfu_1+w_2\bfu_2+w_3\bfu_3$, with $\|\bfx_1\| = 1$. From \cref{thm:optimality_criterion_n=1} it follows that $\XX_1$ is optimal for $\calA$ if and only if $w_2=0$ and
\begin{align}\label{ineq:optimality-3D}
w_3^2 \le \frac{\sigma_2^2 - \sigma_3^2}{\sigma_1^2 - \sigma_3^2}.
\end{align}
Now, let $w_1=\cos(\alpha)$, $w_2=0$ and $w_3=\sin(\alpha)$, where $\alpha$ is the angle between $\XX_1$ and the classical optimal space $\spn\{\bfu_1\}$. Condition \cref{ineq:optimality-3D} is then equivalent to 
\begin{align}\label{ineq:optimality-3D-angle}
|\alpha|\leq \extreme{\alpha}:=\arcsin\left(\sqrt{\frac{\sigma_2^2 - \sigma_3^2}{\sigma_1^2 - \sigma_3^2}}\right).
\end{align}
Thus, $\XX_1$ is optimal for $\calA$ if and only if it is rotated in the $(\bfu_1,\bfu_3)$-plane with an angle less than or equal to $\extreme{\alpha}$ from the $\bfu_1$-axis.
An illustration of this is given in \cref{fig:optimality-3D} for $\bfu_i=\bfe_i$, $i=1,2,3$.
\end{example}

\begin{figure}
\center
\begin{tikzpicture}[scale=2]
	\fill[black!10] (0,0) ellipse (2cm and 1cm);
    \draw[thick] (0,0) ellipse (2cm and 1cm);
    \draw[dashed,thick] (0,0)--(2,0) node[midway,above]{$\sigma_1$};
    \draw[dashed,thick] (0,0)--(0,1) node[midway,left]{$\sigma_3$};
    \draw[thick] (-2,-1)--(2,1) node[right,below]{$\XX_1$};
    \begin{scope}
    \path[clip] (0,0)--(2,1)--(2,0);
    \draw (0,0) circle (0.5cm);
    \node[above right] at (0.24,-0.02) {$\alpha$}; 
    \end{scope}
\end{tikzpicture}
\caption{The $(\bfe_1,\bfe_3)$ cross-section of $\calA$ in the case $m=3$. The space $\XX_1$ is optimal for $\calA$ if and only if $|\alpha|\leq\extreme{\alpha}$ in \cref{ineq:optimality-3D-angle}.}\label{fig:optimality-3D}
\end{figure}

\begin{example} \label{ex:n=1_Hankel_angle}
Similar to an example in \cite{Antoulas:1997} we consider the $3\times3$ matrix 
\begin{equation*} 
A = \begin{bmatrix}
      1 & 0 & 1/4 \\
      0 & 1/4 & 0 \\
      1/4 & 0 & 1
    \end{bmatrix}.
\end{equation*}
Note that this is a symmetric matrix with Hankel structure.
It is easy to verify that $A = U \Sigma U^T$, with
\begin{equation*}
  \sigma_1 = \frac{5}{4},\ 
  \bfu_1 = \frac{1}{\sqrt{2}}\begin{bmatrix} 1 \\ 0 \\ 1 \end{bmatrix}; \quad 
  \sigma_2 = \frac{3}{4},\ 
  \bfu_2 = \frac{1}{\sqrt{2}}\begin{bmatrix} 1 \\ 0 \\ -1 \end{bmatrix}; \quad 
  \sigma_3 = \frac{1}{4},\ 
  \bfu_3 = \begin{bmatrix} 0 \\ 1 \\ 0 \end{bmatrix}.
\end{equation*}
From \cref{ex:n=1_3D} we deduce that any space $\XX_1= \spn\{\bfx_1\}$ is optimal for $\calA$ if and only if it is rotated in the $(\bfu_1,\bfu_3)$-plane with an angle less than or equal to
\begin{equation*}
\extreme{\alpha} = \arcsin\bigl(1/\sqrt{3}\bigr)\approx 35.26^\circ
\end{equation*}
from the $\bfu_1$-axis. The maximum angle $\extreme{\alpha}$ corresponds to the unit vector
\begin{equation*}
 \extreme{\bfx}_1=\frac{\sqrt{2}\bfu_1+ \bfu_3}{\sqrt{3}}
 =\frac{1}{\sqrt{3}}\begin{bmatrix} 1 \\ 1 \\ 1 \end{bmatrix},
\end{equation*}
which will be an interesting choice for structure-preserving approximation (see \cref{ex:n=1_Hankel_range}).
\end{example}

If $A$ is a symmetric matrix, then the low-rank approximations in \cref{thm:low-rank-1,thm:low-rank-2} do not, in general, result in a symmetric approximation to $A$. As we shall see in the next proposition, if given a proper choice of the scaling factor, then each unit vector satisfying the optimality criterion in \cref{thm:optimality_criterion_n=1} provides a symmetric best rank-$1$ approximation to a symmetric matrix $A$ (at least in the case $m=3$). We remark that the next result is very similar to \cite[Theorem 3.1]{Antoulas:1997}. Specifically, if $A$ is a Hankel matrix, then \cite[Theorem 3.1]{Antoulas:1997} provides a characterization of best rank-$1$ approximations to $A$ that preserve the Hankel structure. This characterization was later generalized to best rank-$1$ Hankel approximations to a symmetric matrix $A$ in \cite[Theorem 4.1]{Knirsch:preprint}. 

\begin{proposition} \label{pro:n=1_symmetry}
Let $n=1$ and $m=3$. Let $A$ be a symmetric matrix and let $\bfx_1=\sum_{i=1}^3w_i\bfu_i$ be a unit vector such that $\XX_1:=\spn\{\bfx_1\}$ is optimal for $\calA$. Then, for any $\sigma_{\bfx_1}\in \RR$ such that
\begin{equation}
\label{eq:range}
 \mathfrak{L}(\bfx_1):= \frac{(\sigma_1-\sigma_2)(\sigma_2-\sigma_3)}{(\sigma_2-\sigma_3)-(\sigma_1-\sigma_3)w_3^2}
  \leq \sigma_{\bfx_1} \leq
  \frac{(\sigma_1+\sigma_2)(\sigma_2+\sigma_3)}{(\sigma_2+\sigma_3)+(\sigma_1-\sigma_3)w_3^2}=:\mathfrak{U}(\bfx_1),
\end{equation}
we have
\begin{equation}\label{eq:conjecture}
\|A  - \sigma_{\bfx_1} \bfx_1 \bfx_1^T\|=\sigma_2.
\end{equation}
\end{proposition}
\begin{proof}
Without loss of generality, we can restrict ourselves to the case of $A$ being a diagonal matrix $\Sigma$ and $\bfu_i=\bfe_i$, the elements of the canonical basis.
Proving equality \cref{eq:conjecture} is equivalent to showing that the maximum modulus of the eigenvalues of the matrix
$\Sigma  - \sigma_{\bfx_1} \bfx_1 \bfx_1^T$ is equal to $\sigma_2$. Since $\XX_1$ is optimal for $\calA$, we know from \cref{thm:optimality_criterion_n=1} that $w_2=0$. 
Therefore, the eigenvalues of $\Sigma  - \sigma_{\bfx_1} \bfx_1 \bfx_1^T$ are given by $\sigma_2$ and by the eigenvalues of the submatrix obtained by removing the second row and the second column, i.e.,
\begin{equation}
\label{eq:Mw}
\begin{bmatrix}
\sigma_1-\sigma_{\bfx_1} w_1^2 & -\sigma_{\bfx_1} w_1 w_3 \\
-\sigma_{\bfx_1} w_1 w_3 & \sigma_3-\sigma_{\bfx_1} w_3^2
\end{bmatrix}.
\end{equation}
Then, proving equality \cref{eq:conjecture} is equivalent to showing that the eigenvalues of the matrix in \cref{eq:Mw} are less than or equal to $\sigma_2$ in modulus. A direct computation shows that its two (real) eigenvalues are given by
$$
\lambda_{\pm}=\frac{\sigma_1+\sigma_3-\sigma_{\bfx_1} \pm\sqrt{(\sigma_1-\sigma_3-\sigma_{\bfx_1})^2+ 4w_3^2(\sigma_1-\sigma_3)\sigma_{\bfx_1}}}{2}.
$$
Imposing
$
-\sigma_2\leq \lambda_{\pm}\leq \sigma_2
$
results in the range \cref{eq:range} for $\sigma_{\bfx_1}$.
\end{proof}

Let $m=3$. Recall from \cref{ineq:optimality-3D}--\cref{ineq:optimality-3D-angle} that $\XX_1$ is optimal for $\calA$ if and only if $w_2=0$ and
$$w_3^2 \le \sin^2(\extreme{\alpha}):=\frac{\sigma_2^2 - \sigma_3^2}{\sigma_1^2 - \sigma_3^2}.$$
Set $\extreme{\bfx}_1:=\cos(\extreme{\alpha})\bfu_1+\sin(\extreme{\alpha})\bfu_3$, then one can check that
\begin{equation*}
\mathfrak{L}(\bfx_1)\leq \mathfrak{L}(\extreme{\bfx}_1)=\sigma_1+\sigma_3= \mathfrak{U}(\extreme{\bfx}_1)\leq \mathfrak{U}(\bfx_1),
\end{equation*}
for any $\bfx_1$ such that its span is optimal for $\calA$.
Therefore, the range of values in \cref{eq:range} for the scaling factor $\sigma_{\bfx_1}$ is always non-empty. In particular, it always contains the value $\sigma_1+\sigma_3$. This means that there always exists at least one best low-rank approximation in any optimal space for $\calA$ (with $m=3$ and $n=1$). The classical truncated SVD approximation to $A$ corresponds to $\bfx_1=\bfu_1$, and in this case we have $\sigma_1-\sigma_2\leq \sigma_{\bfu_1}\leq \sigma_1+\sigma_2$. 
This is in agreement with \cref{eq:no-unique-spectral}.

\begin{example} \label{ex:n=1_Hankel_range} 
As a continuation of \cref{ex:n=1_Hankel_angle}, consider again the matrix
\begin{equation}\label{eq:A_Hankel_3}
A = \begin{bmatrix}
      1 & 0 & 1/4 \\
      0 & 1/4 & 0 \\
      1/4 & 0 & 1
    \end{bmatrix}.
\end{equation}
According to \cref{pro:n=1_symmetry}, any choice $\bfx_1=\cos(\alpha)\bfu_1+\sin(\alpha)\bfu_3$, with $|\alpha|\leq\extreme{\alpha}$, leads to a range of best rank-$1$ approximations to $A$ that are symmetric, i.e.,
\begin{equation*}
\sigma_{\bfx_1} \bfx_1 \bfx_1^T = \frac{\sigma_{\bfx_1}}{2} \begin{bmatrix}
 \cos^2(\alpha) & \sqrt{2} \cos(\alpha) \sin(\alpha) & \cos^2(\alpha) \\
 \sqrt{2} \cos(\alpha) \sin(\alpha) & 2 \sin^2(\alpha) & \sqrt{2} \cos(\alpha) \sin(\alpha) \\
 \cos^2(\alpha) & \sqrt{2} \cos(\alpha) \sin(\alpha) & \cos^2(\alpha)
\end{bmatrix},
\end{equation*}
for any $\sigma_{\bfx_1}\in \RR$ such that
\begin{equation*}
 \frac{1}{2-4\sin^2(\alpha)} \leq \sigma_{\bfx_1} \leq \frac{2}{1+\sin^2(\alpha)}.
\end{equation*}
The specific choice $\extreme{\bfx}_1$, corresponding to the maximum angle $\extreme{\alpha}$, gives a best rank-$1$ approximation that even preserves the Hankel structure of $A$, i.e.,
\begin{equation*}
\sigma_{\extreme{\bfx}_1} \extreme{\bfx}_1 \extreme{\bfx}_1^T = \frac{1}{2} \begin{bmatrix}
  1 & 1 & 1 \\
  1 & 1 & 1 \\
  1 & 1 & 1
\end{bmatrix},
\end{equation*}
since $\sigma_{\extreme{\bfx}_1}=3/2$. Similarly, the approximation obtained by taking the angle $-\extreme{\alpha}$ preserves the Hankel structure as well. According to \cite[Theorem~3.1]{Antoulas:1997}, these matrices are the only two Hankel-preserving best rank-$1$ approximations to $A$ in \cref{eq:A_Hankel_3}.
\end{example}

As shown in \cite{Antoulas:1997}, it is not always possible to find a Hankel-preserving best rank-$1$ approximation to a Hankel matrix $A$. When this is not possible one can ask the question of how well one can approximate $A$ with rank-$1$ Hankel matrices, and this has been studied in \cite{Knirsch:preprint}.

\section{Alternative optimality criteria}\label{sec:optimality-crit-alt}
In this section we provide some alternative optimality criteria that are useful in the case of large $n$. While this is not relevant for low-rank approximation, these results are still of independent interest for the Kolmogorov $n$-width. To simplify the exposition, we will in this section only consider matrices $A$ that are of full rank, i.e., $r=m$.
Recall that a necessary condition for an $n$-dimensional
space $\XX_n$ to be optimal for the $n$-width is that
it is orthogonal to $\bfu_{n+1}$ (see \cref{thm:Karlovitz-nec}).
This implies that the only optimal space for $n=m-1$ is given in \cref{eq:optimal:N-1}.

Suppose now that $n \le m-2$ and that
$\XX_n$ is orthogonal to $\bfu_{n+1}$.
Let us denote the orthogonal 
complement of $\XX_n \oplus \bfu_{n+1}$ in $\RR^m$ by $\YY_{m-n-1}$, and suppose that
we can represent it in the form
\begin{equation*}
\YY_{m-n-1}=\spn\{\bfy_1,\ldots,\bfy_{m-n-1}\},
\end{equation*}
where $\bfy_1,\ldots,\bfy_{m-n-1}$ are orthonormal vectors in $\RR^m$.
We can express these vectors as
\begin{equation*}
 \bfy_j = \sum_{i=1}^m q_{ij} \bfu_i, \quad j=1,\ldots,m-n-1,
\end{equation*}
for coefficients $q_{ij} \in \RR$, $j=1,\ldots,m-n-1$,
where now
$$ q_{n+1,j} = 0, \quad j=1,\ldots, m-n-1. $$
Denoting by $Q\in\RR^{m,m-n-1}$ the matrix
\begin{equation}
\label{eq:W-alt}
[q_{ij}]_{i=1,\ldots,m,j=1,\ldots,m-n-1},
\end{equation}
we obtain the following alternative characterization of optimality for $\XX_n$.

\begin{lemma}\label{lem:Eisaneigenvalue_alt_reduce}
Let $Q$ be the matrix in \cref{eq:W-alt}.
The subspace
$\XX_n$ is optimal for ${\calA}$ if and only if $\XX_n \perp \bfu_{n+1}$ and the largest eigenvalue of
\begin{equation*}
Q^T \Sigma^2 Q
\end{equation*}
is at most $\sigma_{n+1}^2$.
\end{lemma}
\begin{proof}
Recall from \cref{eq:max} that
$$
E(\calA, \XX_n) =
   \max_{\bfz \perp \XX_n} \sqrt{\frac{\bfz^TAA^T\bfz}{\bfz^T\bfz}}.
$$
Following the argument of Karlovitz in \cite[Theorem~1]{Karlovitz:76}, 
any $\bfz$ orthogonal to $\XX_n$ can be expressed uniquely as
$\bfz = \bfy \oplus \bfx$, where $\bfy \in \YY_{m-n-1}$
and $\bfx \in \spn\{\bfu_{n+1}\}$.
Then, 
\begin{equation*}
\frac{\bfz^T A A^T \bfz}{\bfz^T \bfz}
   = \frac{\bfx^T \bfx \sigma_{n+1}^2 + \bfy^T A A^T \bfy}
         {\bfx^T \bfx + \bfy^T \bfy}
  \le \max \left\{ \sigma_{n+1}^2, 
   \frac{\bfy^T A A^T \bfy} {\bfy^T \bfy} \right\}, 
\end{equation*}
since it is a convex combination of $\sigma_{n+1}^2$ and $ \frac{\bfy^T A A^T \bfy} {\bfy^T \bfy}$.
We conclude that
\begin{equation*}
E(\calA, \XX_n) =
  \max \left\{ \sigma_{n+1}, 
   \max_{\bfy \in \YY_{m-n-1}} \sqrt{\frac{\bfy^TAA^T\bfy}{\bfy^T\bfy}}
  \right\}. 
\end{equation*}

Any $\bfy \in \YY_{m-n-1}$ can be represented as
$$ \bfy = \sum_{j=1}^{m-n-1} c_j \bfy_j, $$
for coefficients $c_1,\ldots,c_{m-n-1}$.
Setting $\bfc := [c_1,\ldots,c_{m-n-1}]^T$, we have
$$ \bfy^T \bfy = \sum_{j=1}^{m-n-1} c_j^2 = \bfc^T \bfc. $$
Setting
$Y := [\bfy_1,\ldots,\bfy_{m-n-1}]$, we also have
$$ \bfy = Y \bfc = U Q \bfc, $$
and so
\begin{equation*}
 \bfy^T A A^T \bfy = \bfc^T Q^T U^T A A^T U Q \bfc
                     = \bfc^T Q^T \Sigma^2 Q \bfc.
\end{equation*}
Therefore,
\begin{equation*}
 \max_{\bfy \in \YY_{m-n-1}}
      \frac{\bfy^T A A^T \bfy}{\bfy^T \bfy}
   = \max_{\bfc \in \RR^{m-n-1}} 
    \frac{\bfc^T Q^T \Sigma^2 Q\bfc}{\bfc^T\bfc},
\end{equation*}
which is the largest eigenvalue of $Q^T \Sigma^2 Q$.
\end{proof}

Suppose now that $n = m-2$ and that
$\XX_{m-2}$ is orthogonal to $\bfu_{m-1}$.
Let $\YY_1$ be the orthogonal complement to $\XX_{m-2} \oplus \bfu_{m-1}$
in $\RR^m$.
Let $\bfy_1$ be a unit vector in $\YY_1$ (which is unique
up to a change of sign).
We can express $\bfy_1$ in the basis $\bfu_1,\ldots,\bfu_m$,
and write
\begin{equation*}
 \bfy_1 = \sum_{i=1}^m q_i \bfu_i,
\end{equation*}
for coefficients $q_1,\ldots,q_m \in \RR$ such that $\sum_{i=1}^m q_i^2 = 1$ and $q_{m-1} = 0$.

\begin{theorem}\label{thm:N-2}
The subspace $\XX_{m-2}$ is optimal if and only if $q_{m-1}=0$ and 
$$ \sum_{\substack{i=1 \\ i \ne m-1}}^m q_i^2 \sigma_i^2 \le
     \sigma_{m-1}^2. $$
\end{theorem}

\begin{proof}
This is just an application of \cref{lem:Eisaneigenvalue_alt_reduce}
for $n=m-2$, in which case the matrix $Q^T\Sigma^2 Q$
has the single element
\begin{align*}
 \sum_{\substack{i=1 \\ i \ne m-1}}^m q_i^2 \sigma_i^2.
\end{align*}
\end{proof}

\begin{example}\label{ex:n=1_3D_alt}
Let $m=3$ and let $\XX_1$ be a $1$-dimensional subspace of $\RR^3$
that is orthogonal to $\bfu_2$, and let
$\bfy_1 = q_1 \bfu_1 + q_3 \bfu_3$ be a unit vector orthogonal to $\XX_1$.
From \cref{thm:N-2} it follows that $\XX_1$ is optimal
for $\calA$ if and only if
$$ q_1^2 \sigma_1^2 + q_3^2 \sigma_3^2 \le \sigma_2^2. $$
If
$$ \XX_1 = \spn\{\cos(\alpha)\bfu_1+\sin(\alpha)\bfu_3\}, $$
then
$$ \YY_1 = \spn\{-\sin(\alpha)\bfu_1+\cos(\alpha)\bfu_3\}, $$
and $(q_1,q_3) = \pm (-\sin(\alpha),\cos(\alpha))$, thus
the optimality condition can be expressed as
$$ \sin^2(\alpha) \le \frac{\sigma_2^2 - \sigma_3^2}
                          {\sigma_1^2 - \sigma_3^2}. $$
This agrees with \cref{ex:n=1_3D}.
\end{example}

\section{Totally positive matrices}\label{sec:tp-matrices}
Melkman and Micchelli studied the $n$-width problem for a certain class of matrices, and in this section we compare their results with the optimality criteria in \cref{sec:optimality-crit,sec:1-width}. If $A$ is strictly totally positive, i.e., all its minors are positive, then two optimal spaces for $\calA$ are constructed in \cite[Section~4]{Melkman:78}. These two spaces are in general different from the classical optimal space $\spn\{\bfu_1,\ldots,\bfu_n\}$. We will describe the first of these optimal spaces here. The second will be discussed in the next section.

When $A$ is strictly totally positive it follows from a theorem of Gantmacher and Krein \cite{Gantmacher:2002} that the singular values are positive and distinct,
$$\sigma_1>\sigma_2>\dots>\sigma_m>0,$$
and the right singular vectors of $A$ have the following sign properties,
\begin{equation}\label{eq:STP-sign}
S^+(\bfv_{n+1})=S^-(\bfv_{n+1})=n, \quad n=0,\ldots,m-1.
\end{equation}
Here $S^-(\bfv)$ denotes the actual sign changes of the vector $\bfv$, where zero components are discarded and $S^+(\bfv)$ is the maximum number of sign changes obtainable by adding $1$ or $-1$ to the zero components of $\bfv$.
It follows from \cref{eq:STP-sign} that $v_{n+1,1}v_{n+1,m}\neq 0$ and we can assume, without loss of generality, that $v_{n+1,1}>0$. Moreover, using \cref{eq:STP-sign}, there exist indices $0=\ell_0<\ell_1<\dots<\ell_n<\ell_{n+1}=m$, denoting the sign changes in $\bfv_{n+1}$, i.e, such that
$$
v_{n+1,i}(-1)^j\geq 0, \quad \ell_j<i\leq \ell_{j+1},\quad j=0,1,\ldots,n.
$$
To simplify the exposition, let us assume that the vector $\bfv_{n+1}$ has no zero components; see \cite[Section~4]{Melkman:78} for the general case. The index $\ell_j$ is then the index before the sign change, i.e., such that $v_{n+1,\ell_j}v_{n+1,\ell_j+1}<0$.
For each $j=1,2,\ldots,n$, define the $m$-dimensional vector $\bfs_j$ by
\begin{equation*}
s_{j,k}:=\begin{cases}
  1/|v_{n+1,k}|, &k=\ell_j,\ell_{j}+1,
  \\
  0, &\text{otherwise}.
  \end{cases}
\end{equation*}
Then, $\bfs_j\perp\bfv_{n+1}$ for each $j=1,\ldots,n$, and Melkman and Micchelli proved the following result \cite[Theorem~3.1]{Melkman:78}.
\begin{theorem}\label{thm:Melkman}
If $A$ is a strictly totally positive matrix, then
\begin{equation}\label{eq:Melkman1}
\XX_n^1:=\spn\{A\bfs_1,\ldots,A\bfs_n\}
\end{equation}
is an optimal subspace for $\calA:=\{A\bfx: \|\bfx\|\leq 1\}$.
\end{theorem}

As a consequence of the above result, if we use a Gram--Schmidt process to find an orthonormal basis for $\XX_n^1$, then we immediately obtain a best rank-$n$ approximation to $A$ by applying \cref{thm:low-rank-1}.

Note that the space $\XX_n^1$ in \cref{eq:Melkman1} satisfies the necessary condition $\XX_n^1\perp\bfu_{n+1}$ (see \cref{thm:Karlovitz-nec}) since $\bfs_j\perp\bfv_{n+1}$ for each $j=1,\ldots,n$. 

\begin{example}
Consider the case $n=1$ and $m=3$. In view of \cref{thm:optimality_criterion_n=1} and \cref{ex:n=1_3D} it would be interesting to check how far the optimal subspace in \cref{eq:Melkman1} is from the classical space $\spn\{\bfu_1\}$ for different choices of $A$.
Let us take what is perhaps one of the simplest possible choices of a strictly totally positive matrix, the Vandermonde matrix obtained by interpolating at the points $1,2,3$:
\begin{equation*}
A= \begin{bmatrix}
      1 & 1 & 1 \\
      1 & 2 & 4 \\
      1 & 3 & 9 
     \end{bmatrix}.
\end{equation*}
 In this case, it can be checked that the angle between \cref{eq:Melkman1} and the space spanned by $\bfu_1$ is less than $0.171^\circ$, while the maximum angle for an optimal space as in \cref{ex:n=1_3D} is greater than $6.695^\circ$.
\end{example}

\section{Sequence of optimal subspaces}\label{sec:seq-subspaces}
In \cref{thm:optimality_criterion3} we obtained an equivalent condition for optimality that allowed us to classify all optimal spaces of dimension $n=1$ for any matrix $A$ in \cref{thm:optimality_criterion_n=1}. However, as $n$ increases it becomes trickier to apply the optimality criterion in \cref{thm:optimality_criterion3} for an arbitrary matrix $A$. On the other hand, as we saw in the last section, there exist matrices where one can obtain an optimal $n$-dimensional space for $\calA$ using specific properties of the matrix $A$. In this section we prove that, given some initial optimal space $\XX_n^1$, we can obtain a whole sequence of optimal spaces $\XX_n^p$, $p\geq 1$. Moreover, this sequence converges to the classical optimal space as $p\to\infty$. The arguments here hold for any matrix $A$ and are based on those found in \cite{Floater:2017,Floater:2018,Sande:2019} for an integral operator in $L^2$.

Let $\XX_n^1$ and $\YY_n^1$ be any $n$-dimensional subspaces of $\RR^m$, and define the sequence of subspaces $\XX_n^p$ and $\YY_n^p$ by 
\begin{equation}\label{eq:seq}
\XX_n^{p}:=A(\YY_n^{p-1}),\quad \YY_n^{p}:=A^T(\XX_n^{p-1}), \quad p=2,3,\ldots.
\end{equation}
Then, similar to \cite[Lemma~1]{Floater:2017}, we have the following lemma.
\begin{lemma}\label{lem:lifting}
For any matrix $A$ and any subspaces $\XX_n^1$ and $\YY_n^1$, we have
\begin{align*}
 E(\calA, \XX_n^p) &\le E(\calA_T, \YY_n^{p-1}),
 \\
 E(\calA_T, \YY_n^p) &\le E(\calA, \XX_n^{p-1}),
\end{align*}
for all $p\geq 2$.
\end{lemma}
\begin{proof}
The two inequalities are analogous and so we only prove the last one.
Let $P_n$ be the orthogonal projection onto $\XX_n^{p-1}$.
Then, the image of $A^TP_n$ is $\YY_n^p=A^T(\XX_n^{p-1})$ and so
 \begin{align*}
  E(\calA_T,\YY_n^p) \leq \max_{\|\bfx\|\leq1} \|(A^T-A^T P_n)\bfx\| 
    = \max_{\|\bfx\|\leq1} \|(A-P_n A)\bfx\| = E(\calA,\XX_n^{p-1}).
 \end{align*}
\end{proof}

Since $d_n(\calA)=d_n(\calA_T)=\sigma_{n+1}$, we can apply \cref{lem:lifting} in an induction argument on $p$ to obtain the following theorem.
\begin{theorem}\label{thm:optimal-seq}
Suppose the subspace $\XX_n^1$ is optimal for $\calA$ and $\YY_n^1$ is optimal for $\calA_T$. Then, 
\begin{itemize}
 \item the subspaces $\XX_n^p$ in \cref{eq:seq} are optimal for $\calA$, and
 \item the subspaces $\YY_n^p$ in \cref{eq:seq} are optimal for $\calA_T$,
\end{itemize}
for all $p\geq 2$.
\end{theorem} 
\begin{proof}
Assume $\XX_n^{p-1}$ is optimal for $\calA$ and $\YY_n^{p-1}$ is optimal for $\calA_T$. Then, using \cref{lem:lifting}, we have
\begin{align*}
 E(\calA, \XX_n^p) &\le E(\calA_T, \YY_n^{p-1})=d_n(\calA_T)=d_n(\calA),
 \\
 E(\calA_T, \YY_n^p) &\le E(\calA, \XX_n^{p-1})=d_n(\calA)=d_n(\calA_T),
\end{align*}
and so $\XX_n^{p}$ is optimal for $\calA$ and $\YY_n^{p}$ is optimal for $\calA_T$. The result now follows from induction on $p$.
\end{proof}

Note that for $p\geq 2$, the spaces $\XX_n^p$ and $\YY_n^p$ could in general have dimension less than $n$, but they are still optimal for the $n$-width problem whenever $\XX_n^1$ and $\YY_n^1$ are optimal. In fact, if $\XX_n^p$ has dimension $k$, $0\leq k<n$, then $d_k(\calA)$ must equal $d_n(\calA)$ by definition of the $n$-width.

\begin{example}
Let $A$ be a strictly totally positive matrix. Then, by definition, $A^T$ is also strictly totally positive, and if we construct the vectors $\bft_j$, $j=1,\ldots,n$, in a way analogous to the $\bfs_j$ in the previous section, it follows from \cref{thm:Melkman} that
\begin{equation*}
\YY_n^1:=\spn\{A^T\bft_1,\ldots,A^T\bft_n\}
\end{equation*}
is optimal for $\calA_T$. Using \cref{thm:optimal-seq} we then have that, for $p\geq 1$, the spaces
\begin{equation}\label{eq:optimal-seq-TP}
\XX_n^p=\begin{cases}
  \spn\{(AA^T)^i A\bfs_1,\ldots, (AA^T)^i A\bfs_n\}, & p = 2i+1,
  \\
  \spn\{(AA^T)^{i+1}\bft_1,\ldots, (AA^T)^{i+1}\bft_n\}, & p=2i+2,
  \end{cases}
\end{equation}
are optimal for $\calA$. Moreover, we can apply \cref{thm:low-rank-1} to an orthonormal basis for any of the above subspaces $\XX_n^p$, $p\geq 1$, to obtain a best rank-$n$ approximation to $A$. Similarly for $\YY_n^p$ and \cref{thm:low-rank-2}.
We remark that the space $\XX_n^2$ in \cref{eq:optimal-seq-TP} is the second optimal space found by Melkman and Micchelli. 
\end{example}

\begin{example}
Let us compare the result of \cref{thm:optimal-seq} with the optimality criteria in \cref{sec:1-width}. For simplicity we consider the case $n=1$, $m=3$ and $A=\Sigma$. We further assume that the unit vector $\bfx_1$ is at the boundary of satisfying the optimality criteria in \cref{sec:1-width}. More precisely, we let $\bfx_1=\sum_{j=1}^3w_j\bfu_j$, and using \cref{ineq:optimality-3D}, we assume that
\begin{align*}
w_1^2 =  \frac{\sigma_1^2 - \sigma_2^2}{\sigma_1^2 - \sigma_3^2}, 
\quad w_2=0,
\quad w_3^2 = \frac{\sigma_2^2 - \sigma_3^2}{\sigma_1^2 - \sigma_3^2}.
\end{align*}
It then follows from \cref{thm:optimality_criterion_n=1} that $\spn\{\bfx_1\}$ is optimal for $\calA$.
Now, let $\bfy_1=A\bfx_1/\|A\bfx_1\|$. From \cref{thm:optimal-seq} we know that $\spn\{\bfy_1\}$ is also optimal for $\calA$. Moreover, if we let $\bfy_1=\sum_{j=1}^3z_j\bfu_j$, then $z_2=0$ and 
\begin{align*}
z_3^2 = \frac{\sigma_3^2w_3^2}{\sigma_1^2w_1^2+ \sigma_3^2w_3^2} = \frac{\sigma_2^2 - \sigma_3^2}{\sigma_1^2 - \sigma_3^2 + s}< \frac{\sigma_2^2 - \sigma_3^2}{\sigma_1^2 - \sigma_3^2} = w_3^2,
\end{align*}
where $s=(\sigma_1^2/\sigma_3^2-1)(\sigma_1^2-\sigma_2^2)> 0$. Thus, $\bfy_1$ is closer to the first singular vector (or in this case, eigenvector) $\bfu_1=\bfe_1$ than $\bfx_1$. We will look closer at this property in the next theorem.
\end{example}

Note that the definition of the spaces $\XX_n^p$ and $\YY_n^p$ in \cref{eq:seq} is very similar to the (block) power method for eigenvalue approximation. The following result, based on \cite[Theorem~7.1]{Sande:2019}, should therefore not come as a surprise for anyone familiar with this method.
\begin{theorem}\label{thm:eig}
Suppose $\XX_n^1$ is optimal for $\calA$ and $\YY_n^1$ is optimal for $\calA_T$. Let $P_{n,p}$ be the orthogonal projection onto $\XX_n^p$ and $\Pi_{n,p}$ be the orthogonal projection onto $\YY_n^p$. Then,
\begin{align*}
\|(I-P_{n,p})\bfu_j\|, \|(I-\Pi_{n,p})\bfv_j\| \leq \left(\frac{\sigma_{n+1}}{\sigma_j}\right)^p, \quad j=1,2,\ldots,n,
\end{align*}
and consequently, 
\begin{align*}
\XX_n^p \xrightarrow[p\to\infty]{}\spn\{\bfu_1,\ldots\bfu_n\}, \quad \YY_n^p \xrightarrow[p\to\infty]{}\spn\{\bfv_1,\ldots\bfv_n\}.
\end{align*}
\end{theorem}

The above result follows from the next lemma and so we will postpone the proof.
To ease notation we define the two function classes $\calA^p$ and $\calA^p_T$, for $p\geq 1$, by $\calA^1:=\calA$, $\calA^1_T:=\calA_T$ and
\begin{equation}\label{eq:Ar}
\calA^p:=A(\calA^{p-1}_T),\quad \calA^p_T:=A^T(\calA^{p-1}),
\end{equation}
for $p\geq 2 $.
Using an argument similar to the proofs of \cite[Lemma~1]{Floater:2018} and \cite[Lemma~2]{Sande:2020} we have the following result.
\begin{lemma}\label{lem:seq-sigma}
If $\XX_n^1$ is optimal for $\calA$ and $\YY_n^1$ is optimal for $\calA_T$, then
\begin{align*}
 E(\calA^p, \XX_n^p) =  E(\calA^p_T, \YY_n^p) = (\sigma_{n+1})^p.
\end{align*}
\end{lemma}
\begin{proof}
 Let $P_{n,p}$ be the orthogonal projection onto $\XX_n^p$ and $\Pi_{n,p}$ be the orthogonal projection onto $\YY_n^p$. Then, the matrix
 \begin{align*}
 (I-P_{n,p})A\Pi_{n,p-1} = 0,
 \end{align*}
since $A\Pi_{n,p-1}\bfx\in \XX_n^p$ for any vector $\bfx\in\RR^m$. If we now let the matrix $B$ be defined by $B:=A^T(AA^T)^i$ for $p=2i+2$ and $B:=(A^TA)^i$ for $p=2i+1$, then
 \begin{align*}
 E(\calA^p, \XX_n^p)  &= \|(I-P_{n,p})AB\| =  \|(I-P_{n,p})A(I-\Pi_{n,p-1})B\|
 \\
 &\leq  \|(I-P_{n,p})A\|\,\|(I-\Pi_{n,p-1})B\| = \sigma_{n+1}\,E(\calA_T^{p-1}, \YY_n^{p-1}),
 \end{align*}
 since $\XX_n^p$ is optimal for $\calA$ by \cref{thm:optimal-seq}. By a similar argument we have 
 \begin{align*}
  E(\calA_T^p, \YY_n^p)  = \sigma_{n+1}\,E(\calA^{p-1}, \XX_n^{p-1}),
 \end{align*}
 and the result follows from induction on $p$. 
\end{proof}

From the definitions of $\calA^p$ and $\calA^p_T$ in \cref{eq:Ar} we deduce that $d_n(\calA^p)=d_n(\calA^p)=(\sigma_{n+1})^p$. It thus follows from \cref{lem:seq-sigma} that if $\XX_n^1$ is optimal for $\calA$ and $\YY_n^1$ is optimal for $\calA_T$ then  $\XX_n^p$ is optimal for $\calA^p$ and $\YY_n^p$ is optimal for $\calA_T^p$. In fact, using the arguments of \cite[Section~4]{Floater:2018} one can show that if $\XX_n^1$ is optimal for $\calA$ and $\YY_n^1$ is optimal for $\calA_T$ then  $\XX_n^p$ is optimal for $\calA^s$ and $\YY_n^p$ is optimal for $\calA_T^s$ for all $p\geq s\geq 1$.

\begin{proof}[Proof of \cref{thm:eig}]
The two cases are analogous and so we only consider the case $\|(I-P_{n,p})\bfu_j\|$. Using the definition of the spectral norm and \cref{lem:seq-sigma} we have
\begin{equation}\label{ineq:optimal}
\begin{aligned}
\|(I-P_{n,p})(AA^T)^i\bfx\| &\leq \|(I-P_{n,p})(AA^T)^i\|\,\|\bfx\| \\
&= E(\calA^p, \XX_n^p) = (\sigma_{n+1})^p, &&\quad p=2i,
\\
\|(I-P_{n,p})A(A^TA)^i\bfx\| &\leq \|(I-P_{n,p})A(A^TA)^i\|\,\|\bfx\| \\
&= E(\calA^p, \XX_n^p)  = (\sigma_{n+1})^p, &&\quad p=2i+1,
\end{aligned}
\end{equation}
for any unit vector $\bfx\in \RR^m$. We first consider $p=2i$. Then, for any $j=1,\ldots,n$ we have 
\begin{align*}
\|(I-P_{n,p})\bfu_j\| = \|(I-P_{n,p})\frac{1}{\sigma_j^p}(AA^T)^i\bfu_j\| = \frac{1}{\sigma_j^p}\|(I-P_{n,p})(AA^T)^i\bfu_j\|,
\end{align*}
and by letting $\bfx=\bfu_j$ in \cref{ineq:optimal} we obtain
\begin{align*}
\|(I-P_{n,p})\bfu_j\| \leq  \left(\frac{\sigma_{n+1}}{\sigma_j}\right)^p.
\end{align*}
A similar argument proves the case $p=2i+1$.
\end{proof}

\section{Conclusions} \label{sec:conclusion}
We have addressed the problem of best rank-$n$ approximations to a given matrix $A$ in the spectral norm, and
we have shown that the problem can be related to the concept of Kolmogorov $n$-widths and corresponding optimal spaces.
More precisely, any orthonormal basis in an optimal $n$-dimensional  space for the image of the Euclidean unit ball under $A$ generates a best rank-$n$ approximation to $A$. This results in a variety of best low-rank approximations that are different from the truncated SVD.

In this perspective, we have laid out explicit characterizations of optimal subspaces of any dimension, and presented a complete description of all the optimal one-dimensional subspaces. Furthermore, we have provided a simple construction to obtain a sequence of optimal $n$-dimensional subspaces once an initial optimal subspace is known.

The paper features an explicit theoretical contribution. The task to retrieve useful information while maintaining the underlying physical feasibility often necessitates the search for low-rank approximations with/without specific properties/structures of the data matrix \cite{Antoulas:1997,Chu:2003,Higham:1989,Markovsky:2008,Ottaviani:2014}. In this context, the results we have presented may also have a practical impact. However, we have not considered here the problem of finding efficient algorithms to compute our approximations. We note, on the other hand, that in the special case of Hankel matrices such algorithms have been considered in \cite{Knirsch:preprint}.

\section*{Acknowledgements}
This work was supported by the Beyond Borders Programme of the University of Rome
 Tor Vergata through the project ASTRID (CUP E84I19002250005) and by the MIUR Excellence
Department Project awarded to the Department of Mathematics, University of
 Rome Tor Vergata (CUP E83C18000100006). C. Manni, E. Sande and H. Speleers are members of
Gruppo Nazionale per il Calcolo Scientifico, Istituto Nazionale di Alta Matematica.


\bibliography{nwidths}

\end{document}